\documentclass[11pt]{article}

\usepackage{fullpage}
\usepackage{url}

\usepackage{setspace}

\usepackage{todonotes}

\usepackage{enumerate}
\usepackage[round,comma]{natbib}
\usepackage{amsthm}

\usepackage{pdfsync}
\usepackage{graphicx} % more modern
\usepackage{algorithm}
\usepackage{algorithmic}

\usepackage{amssymb}

\usepackage{amsmath} % Learn about the AMS package, again very useful!

\usepackage{color}
\usepackage{graphicx}
\usepackage{algorithm,algorithmic}
\usepackage{verbatim}

\usepackage{nicefrac}

\newif\ifcomm
\commtrue

\newif\iflong
\longtrue
\newtheorem{thm}{Theorem}

\newtheorem{lem}[thm]{Lemma}
\newtheorem{lemma}[thm]{Lemma}

\newtheorem{defn}{Definition}
\newtheorem{rem}[thm]{Remark}

\newcounter{assumption}%[section]
\renewcommand{\theassumption}{A\arabic{assumption}}

\newenvironment{ass}[1][]{\begin{trivlist}\item[] \refstepcounter{assumption}%
 {\bf Assumption\ \theassumption\ {\em (#1)} } }{%\par\nobreak\noindent\sl\ignorespaces}{%
 \ifvmode\smallskip\fi\end{trivlist}}

\newcommand{\norm}[1]{\left\Vert#1\right\Vert}

\newcommand{\abs}[1]{\left\vert#1\right\vert}

\newcommand{\Real}{\mathbb R}                        % Real numbers
\newcommand{\EE}[1]{{\mathbb E}\left[#1\right]}      % Expectations
\newcommand{\one}[1]{{\mathbb I}_{\{#1\}}}           % Characteristic function

\newcommand{\ra}{\rightarrow}

\newcommand{\argmin}{\mathop{\rm argmin}}

\newcommand{\eqdef}{\stackrel{\mbox{\rm\tiny def}}{=}}

\newcommand{\beq}{\begin{equation}}
\newcommand{\eeq}{\end{equation}}
\newcommand{\beqa}{\begin{eqnarray}}
\newcommand{\eeqa}{\end{eqnarray}}
\newcommand{\beqan}{\begin{eqnarray*}}
\newcommand{\eeqan}{\end{eqnarray*}}
\newcommand{\ben}{\begin{eqnarray*}}
\newcommand{\een}{\end{eqnarray*}}
\newcommand{\bea}{\begin{align*}}
\newcommand{\eea}{\end{align*}}

\ifcomm
   \newcommand\comm[1]{\textcolor{blue}{ #1}}
   \newcommand{\mtodo}[2]{\todo{{\bf #1}: #2}} % To add comments into the text; the first argument is "who", the second is "what"
   \def\here#1{{\bf $\langle\langle$#1$\rangle\rangle$}}

\else
   \newcommand\comm[1]{}
   \newcommand{\mtodo}[2]{}
   \def\here#1{}
\fi

\newcommand{\cZ}{{\cal Z}}
\newcommand{\cX}{{\cal X}}

\newcommand{\cA}{{\cal A}}

\renewcommand{\phi}{\varphi}

\newcommand{\cS}{{\cal S}}

\newcommand{\cN}{{\cal N}}

\title{Linear Programming for Large-Scale Markov Decision Problems}

\author{
Yasin Abbasi-Yadkori\\
Queensland University of Technology\\
\texttt{yasin.abbasiyadkori@qut.edu.au} \\
\and
Peter L. Bartlett \\
UC Berkeley and QUT \\
\texttt{bartlett@eecs.berkeley.edu} \\
\and
Alan Malek \\
UC Berkeley \\
\texttt{malek@eecs.berkeley.edu} \\
}

\begin{document}

\maketitle

\begin{abstract}
We consider the problem of controlling a Markov decision
process (MDP) with a large state space, so as to minimize average cost.
Since it is intractable to compete with the optimal policy for large
scale problems, we pursue the more modest goal of competing with a
low-dimensional family of policies. We use the dual linear programming
formulation of the MDP average cost problem, in which the variable is
a stationary distribution over state-action pairs, and we consider a
neighborhood of a low-dimensional subset of the set of stationary
distributions (defined in terms of state-action features) as
the comparison class.
We propose two techniques, one based on stochastic convex optimization,
and one based on constraint sampling. In both cases, we give bounds
that show that the performance of our algorithms approaches the best
achievable by any policy in the comparison class. Most importantly,
these results depend on the size of the comparison class, but not
on the size of the state space.
Preliminary experiments
show the effectiveness of the proposed algorithms in a queuing
application.
\end{abstract}

%\begin{keywords}
%Online Learning, Markov Decision Processes, Regret Analysis 
%\end{keywords}

\section{Introduction}

We study the average loss Markov decision process problem. The
problem is well-understood when the state and action spaces are small
\citep{Bertsekas-2007}. Dynamic programming (DP) algorithms, such as
value iteration~\citep{Bellman-1957} and policy
iteration~\citep{Howard-1960}, are standard techniques to compute the
optimal policy. In large state space problems, exact DP is not
feasible as the computational complexity scales quadratically with the
number of states.

A popular approach to large-scale problems is to restrict the search
to the linear span of a small number of features. The objective is to
compete with the best solution within this comparison class. Two popular
methods are Approximate Dynamic Programming (ADP) and Approximate Linear
Programming (ALP). This paper focuses on ALP.
For a survey on theoretical results for ADP
see \citep{Bertsekas-Tsitsiklis-1996, Sutton-Barto-1998},
\citep[Vol. 2, Chapter 6]{Bertsekas-2007}, and more recent papers
\citep{Sutton-Szepesvari-Maei-2009,
Sutton-Maei-Precup-Bhatnagar-Silver-Szepesvari-Wiewiora-2009,
Maei-Szepesvari-Bhatnagar-Precup-Silver-Sutton-2009,
Maei-Szepesvari-Bhatnagar-Sutton-2010}. 

Our aim is to develop methods that find policies with
performance guaranteed to be close to the best in the
comparison class but with computational complexity that
does not grow with the size of the state space.
All prior work on ALP either scales badly or requires
access to samples from a distribution that depends on the optimal
policy.

This paper proposes a new algorithm to solve the Approximate Linear Programming problem that is computationally efficient and does not require knowledge of the optimal policy. In particular, we introduce new proof techniques and tools for average cost MDP problems and use these techniques to derive a reduction to stochastic convex optimization with accompanying error bounds. We also propose a constraint sampling technique and obtain performance guarantees under an additional assumption on the choice of features.

\subsection{Notation}
Let $X$ and $A$ be positive integers. Let $\cX= \{1,2,\ldots,X\}$ and $\cA = \{1,2,\ldots, A\}$ be state and action spaces, respectively. 
%Let $\cX$ be a finite state space with cardinality $X$ and $\cA$ be a finite action space with cardinality $A$. 
Let $\Delta_S$ denote probability distributions over set $S$. A policy $\pi$ is a map from the state space to $\Delta_{\cA}$: $\pi:\cX\ra\Delta_{\cA}$. We use $\pi(a|x)$ to denote the probability of choosing action $a$ in state $x$ under policy $\pi$. A transition probability kernel (or transition kernel)  $P:\cX \times \cA \ra \Delta_{\cX}$ maps from the direct product of the state and action spaces to $\Delta_{\cX}$. Let $P^{\pi}$ denote the probability transition kernel under policy $\pi$. A loss function is a bounded real-valued function over state and action spaces, $\ell:\cX\times \cA \ra [0,1]$.  Let $M_{i,:}$ and $M_{:,j}$ denote $i$th row and $j$th column of matrix $M$ respectively. Let $\norm{v}_{1,c} = \sum_{i} c_i\abs{v_i}$ and $\norm{v}_{\infty, c} = \max_i c_i \abs{v_i}$ for a positive vector $c$. We use $\mathbf 1$ and $\mathbf 0$ to denote vectors with all elements equal to one and zero, respectively. We use $\wedge$ and $\vee$ to denote the minimum and the maximum, respectively. For vectors $v$ and $w$, $v \le w$ means element-wise inequality, i.e. $v_i \le w_i$ for all $i$.

\subsection{Linear Programming Approach to Markov Decision Problems}

Under certain assumptions, there exist a scalar $\lambda_*$ and a vector $h_*\in \Real^X$ that satisfy the Bellman optimality equations for average loss problems, 
\[
\lambda_* + h_*(x) = \min_{a \in \cA} \left[ \ell(x,a) + \sum_{x'\in\cX} P_{(x,a),x'} h_*(x') \right] \; .
\] 
The scalar $\lambda_*$ is called the optimal average loss, while the vector $h_*$ is called a differential value function. The action that minimizes the right-hand side of the above equation is the optimal action in state $x$ and is denoted by $a_*(x)$. The optimal policy is defined by $\pi_*(a_*(x) | x)=1$. Given $\ell$ and $P$, the objective of the \textit{planner} is to compute the optimal action in all states, or equivalently, to find the optimal policy.

The MDP problem can also be stated in the LP formulation~\citep{Manne-1960},
\begin{align}
\label{eq:primal}
&\max_{\lambda, h} \lambda\,, \\
\notag
&\mbox{s.t.}\quad B(\lambda \mathbf 1 + h) \le \ell + P h \,,
\end{align}
where $B\in \{0,1\}^{XA\times X}$ is a binary matrix such that the $i$th column has $A$ ones in rows $1+(i-1)A$ to $iA$.  Let $v_{\pi}$ be the stationary distribution under policy $\pi$ and let $\mu_{\pi}(x,a) = v_{\pi}(x) \pi(a|x)$. We can write
\begin{align*}
\pi_* &= \argmin_{\pi } \sum_{x\in \cX} v_{\pi}(x) \sum_{a\in\cA} \pi(a|x) \ell(x,a) \\ 
&= \argmin_{\pi}\sum_{(x,a)\in \cX\times \cA} \mu_{\pi}(x,a) \ell(x,a) = \argmin_{\pi }\mu_\pi^\top \ell \; .
\end{align*}
In fact, the dual of LP~\eqref{eq:primal} has the form of
\begin{align}
\label{eq:dual}
&\min_{\mu\in \Real^{X A}} \mu^\top \ell\,, \\
\notag
&\mbox{s.t.}\quad \mu^\top \mathbf 1 = 1,\, \mu \geq \mathbf 0,\, \mu^\top (P - B) = \mathbf 0 \; .
\end{align}
The objective function, $\mu^\top \ell$, is the average loss under
stationary distribution $\mu$. The first two constraints ensure that
$\mu$ is a probability distribution over state-action space, while the
last constraint ensures that $\mu$ is a stationary distribution. Given
a solution $\mu$, we can obtain a policy via
$\pi(a|x) = \mu(x,a)/\sum_{a'\in\cA} \mu(x,a')$.

%Consider a space of policies. Find the best among them... Feasible set might be empty. 

\subsection{Approximations for Large State Spaces}
\label{sec:ALP}

The LP formulations \eqref{eq:primal} and \eqref{eq:dual} are not
practical for large scale problems as the number of variables and
constraints grows linearly with the number of states.
\citet{Schweitzer-Seidmann-1985} propose approximate linear
programming (ALP) formulations. These methods were later improved by
\citet{DeFarias-VanRoy-2003, deFarias-VanRoy-NIPS-2003,
Hauskrecht-Kveton-2003, Guestrin-Hauskrecht-Kveton-2004,
Petrik-Zilberstein-2009, Desai-Farias-Moallemi-2012}. As noted by
\citet{Desai-Farias-Moallemi-2012}, the prior work on ALP either
requires access to samples from a distribution that depends on optimal
policy or assumes the ability to solve an LP with as many constraints
as states. (See Section~\ref{sec:related-work} for a more detailed discussion.) 
Our objective is to design algorithms for very large MDPs that do not
require knowledge of the optimal policy.
%and show performance bounds without imposing such strong conditions. 
%However, even under these idealized assumptions, the existing theoretical results have several limitations.

In contrast to the aforementioned methods, which solve the primal ALPs (with value
functions as variables), we work with the dual form \eqref{eq:dual} (with stationary distributions as variables). Analogous to ALPs, we control the complexity by limiting our search to a linear subspace defined by a small number of {\em features}. Let $d$ be the number of features and $\Phi$ be a $(X A)\times d$ matrix with features as column vectors. By adding the constraint $\mu = \Phi \theta$, we get
\begin{align*}
%\label{eq:dual-apprx}
&\min_{\theta} \theta^\top \Phi^\top \ell\,, \\
\notag
&\mbox{s.t.}\quad \theta^\top \Phi^\top \mathbf 1 = 1,\, \Phi \theta \geq \mathbf 0,\, \theta^\top \Phi^\top (P - B) = \mathbf 0 \; .
\end{align*}
If a stationary distribution $\mu_0$ is known, it can be added to the
linear span to get the ALP
\begin{align}
\label{eq:dual-apprx}
&\min_{\theta} (\mu_0 + \Phi \theta)^\top  \ell\,, \\
\notag
&\mbox{s.t.}\quad (\mu_0+\Phi\theta)^\top  \mathbf 1 = 1,\, \mu_0+ \Phi \theta \geq \mathbf 0,\, (\mu_0+\Phi\theta)^\top (P - B) = \mathbf 0 \; .
\end{align}
%%% Huh? \theta=0 is always feasible!???!
% It is likely that the feasible set of the above ALP is empty, as the span of features might contain no stationary distributions.
Although $\mu_0+\Phi\theta$ might not be a stationary distribution, it still defines a policy\footnote{We use the notation $[v]_{-}=v \wedge 0$ and $[v]_+ = v \vee 0$.} 
\begin{equation}
\label{eqn:policy-definition}
\pi_{\theta}(a|x) = \frac{[\mu_0(x,a) + \Phi_{(x,a),:} \theta]_{+}}{\sum_{a'} [\mu_0(x,a') + \Phi_{(x,a'),:} \theta]_{+}} \,,
\end{equation}
We denote the stationary distribution of this policy $\mu_{\theta}$ which is only equal to $\mu_0+\Phi\theta$ if $\theta$ is in the feasible set. %To emphasize, we prove bounds for the average loss of policy $\pi_\theta$. 
%Our objective is to find a $\theta$ in a bounded set such that $\mu_\theta^\top \ell$ is small. Computing approximate solutions for this problem is the subject of this paper. 

\subsection{Problem definition}
With the above notation, we can now be explicit about the problem we are solving.
\begin{defn}[Efficient Large-Scale Dual ALP]
\label{defn:ELALP}
For an MDP specified by $\ell$ and $P$, the efficient
large-scale dual ALP problem is to produce parameters $\widehat\theta$ such
\begin{equation}\label{eqn:ELALP}
  \mu_{\widehat\theta}^\top\ell
  \leq \min\left\{ \mu_{\theta}^\top\ell : \text{$\theta$ feasible
  for \eqref{eq:dual-apprx}}\right\} + O(\epsilon)
\end{equation}
in time polynomial in $d$ and $1/\epsilon$. The model of computation
allows access to arbitrary entries of $\Phi$, $\ell$,
$P$, $\mu_0$, $P^\top\Phi$, and $\textbf{1}^\top\Phi$ in unit time.
\end{defn}
Importantly, the computational complexity cannot scale with $X$ and we do not assume
any knowledge of the optimal policy. In fact, as we shall see, we
solve a harder problem, which we define as follows.
\begin{defn}[Expanded Efficient Large-Scale Dual ALP]
\label{defn:E.ELALP}
Let $V:\Re^d\to\Re_+$ be some ``violation function'' that
represents how far $\mu_0+\Phi\theta$ is from a valid stationary
distribution, satisfying $V(\theta)=0$ if $\theta$ is a feasible point
for the ALP~\eqref{eq:dual-apprx}. The expanded efficient large-scale
dual ALP
problem is to produce parameters $\widehat\theta$ such that
\begin{equation}
\label{eqn:E-ELALP}
\mu_{\widehat\theta}^\top\ell \leq \min\left\{\mu_{\theta}^\top\ell
+\frac{1}{\epsilon}V(\theta) : \theta\in\Re^d \right\}+O(\epsilon),
\end{equation}
in time polynomial in $d$ and $1/\epsilon$, under the same model of
computation as in Definition~\ref{defn:ELALP}.
\end{defn}
Note that the expanded problem is strictly more general as guarantee
\eqref{eqn:E-ELALP} implies guarantee \eqref{eqn:ELALP}. Also, many
feature vectors $\Phi$ may not admit any feasible points. In this
case, the dual ALP problem is trivial, but the expanded problem is
still meaningful.

Having access to arbitrary entries of the quantities in
Definition~\ref{defn:ELALP} arises naturally in many situations.
%Often, analytic forms of single entries of $\ell$ and $P$ are easy to compute. 
In many cases, entries of $P^\top\Phi$ are easy to compute. 
For example, suppose that for any state $x'$ there are a small number of state-action pairs $(x,a)$ such that $P(x'|x,a)>0$. Consider Tetris; although the number of board configurations is large, each state has a small number of possible neighbors. Dynamics specified by graphical models with small connectivity also satisfy this constraint. Computing entries of $P^\top\Phi$ is also feasible given reasonable features. If a feature $\phi_i$ is a stationary distribution, then $P^\top\phi_i=B^\top \phi_i$. Otherwise, it is our prerogative to design sparse feature vectors, hence making the multiplication easy. We shall see an example of this setting later.

%Our algorithms are stochastic methods that sample from the state-action space to deal with the constraints. Let $q_1$ be a distribution over the state-action space and $q_2$ be a distribution over the state space that satisfy certain conditions (to be specified later). Throughout this paper, we assume we have access to an efficient method that can sample from $q_1$ and $q_2$. %a distribution $q_1$ over the state-action space and a distribution $q_2$ over the state space. Define

\subsection{Our Contributions}

In this paper, we introduce an algorithm that solves the expanded efficient large-scale dual ALP problem under a (standard) assumption that any policy converges quickly to its stationary distribution.

Our algorithm take as input a constant $S$ and an error tolerance
$\epsilon$, and has access to 
%entries of an optional stationary distribution
%$\mu_0$, a feature vector $\Phi$, a loss vector $\ell$, a transition matrix $P$, and 
the various products listed in
Definition~\ref{defn:ELALP}.
Define $\Theta = \{ \theta \ : \ \theta^\top
\Phi^\top \mathbf 1 = 1-\mu_0^\top \mathbf 1,\, \norm{\theta} \le S
\}$.
If no stationary distribution is known, we can simply choose
$\mu_0=\mathbf 0$. The algorithm is based on stochastic convex
optimization. We prove that for any $\delta\in (0,1)$, after
$O(1/\epsilon^4)$ steps of gradient descent, the algorithm finds a
vector $\widehat \theta \in \Theta$ such that, with probability at
least $1-\delta$,
\begin{align*}
\mu_{\widehat\theta}^\top \ell \le& \mu_\theta^\top \ell + \frac{1}{\epsilon}  \norm{[\mu_0 + \Phi\theta]_{-}}_1 + \frac{1}{\epsilon} \norm{(P - B)^\top (\mu_0 + \Phi\theta)}_1 + O(\epsilon \log(1/\delta)) \;
\end{align*}
holds for all $\theta\in\Theta$; i.e., we solve the expanded problem for $V(\theta)$ equal to the $L_1$ error of the violation.
The second and third terms are zero for feasible points (points in the
intersection of the feasible set of LP~\eqref{eq:dual} and the span of
the features).
% When the feasible set is empty,
For points outside the feasible set, these terms measure the
extent of constraint violations for the vector $\mu_0 + \Phi\theta$, which
indicate how well stationary distributions can be represented by the
chosen features.
% Thus, we have a performance bound even if the
% feasible set is empty, which is another way the expanded dual ALP
% problem is more general than the regular version.

%The second assumption is satisfied when, for example, any state has a small number of  neighboring states (to be made precise later). The third assumption is satisfied if, for example, features are nearly uniform, or are all exponential functions. 

The above performance bound scales with $1/\epsilon$ that can be large when the feasible set is empty and $\epsilon$ is very small. We propose a second approach and show that this dependence can be eliminated if we have extra information about the MDP. Our second approach is based on the constraint sampling method that is already applied to large-scale linear programming and MDP problems \citep{DeFarias-VanRoy-2004, Calafiore-Campi-2005, Campi-Garatti-2008}. We obtain performance bounds, but under the condition that a suitable function that controls the size of constraint violations is available. Our proof technique is different from previous work and gives stronger performance bounds. 

Our constraint sampling method takes two extra inputs: functions $v_1$ and $v_2$ that specify the amount of constraint violations that we tolerate at each state-action pair.  The algorithm samples $O(\frac{d}{\epsilon} \log\frac{1}{\epsilon})$ constraints and solves the sampled LP problem. Let $\widetilde \theta$ denote the solution of the sampled ALP and $\theta_{*}$ denote the solution of the full ALP subject to constraints $v_1$ and $v_2$. We prove that with high probability,
\[
\ell^\top \mu_{\widetilde \theta} \le \ell^\top \mu_{\theta_*} + O(\epsilon + \norm{v_1}_1 + \norm{v_2}_1) \; .
\]

%To compare the above methods, consider the case when the feasible set is not empty. Then the stochastic optimization method finds an $\epsilon$-optimal solution after $O(d/\epsilon^4)$ computation. We get a similar complexity bound for a constraint sampling method that uses an Interior Point method to solve the sampled LP~\citep{Potra-Wright-2000}. There is however no $1/\epsilon$ term in the performance bound of the constraint sampling method, but the algorithm requires appropriate constraint violation functions $v_1$ and $v_2$. 

%In our algorithm, we estimate the subgradient by sampling constraints of the LP. A natural question to ask is if we can first sample constraints then exactly solve the resulting LP. Analysis for such an algorithm is presented in Section~\ref{sec:const-sampling}. However the analysis requires stronger conditions on the choice of feature vectors. %and does not solve the expanded version of the efficient large-state ALP problem.

\section{Related Work}
\label{sec:related-work}

\citet{DeFarias-VanRoy-2003} study the \textit{discounted} version of the primal form \eqref{eq:primal}. Let $c\in \Real^X$ be a vector with positive components and $\gamma\in (0,1)$ be a discount factor. Let $L:\Real^X \ra \Real^X$ be the Bellman operator defined by $(L J)(x) = \min_{a\in\cA} (\ell(x,a) + \gamma \sum_{x'\in\cX} P_{(x,a),x'} J(x') )$ for $x\in\cX$. 
Let $\Psi\in\Real^{X \times d}$ be a feature matrix. The exact and approximate LP problems are as follows:
\begin{align*}
&\max_{J\in \Real^X} c^\top J\,, &\max_{w\in\Real^d}\ & c^\top \Psi w\,,  \\
&\mbox{s.t.}\quad L J \ge J\,, &\mbox{s.t.}\quad & L \Psi w \ge \Psi w \; .
\end{align*}
which can also be written as
\begin{align}
\label{eq:ALP-primal-1}
&\max_{J\in \Real^X} c^\top J\,, &\max_{w\in\Real^d}\ & c^\top \Psi w\,,  \\
\notag
&\mbox{s.t.}\quad \forall (x,a),\, \ell(x,a) + \gamma P_{(x,a),:} J \ge J(x)\,, &\mbox{s.t.}\quad & \forall (x,a),\, \ell(x,a) + \gamma P_{(x,a),:} \Psi w \ge (\Psi w)(x) \; .
\end{align}
The optimization problem on the RHS is an approximate LP with the choice of $J = \Psi w$. Let $J_\pi (x) = \EE{\sum_{t=0}^\infty \gamma^t \ell(x_t, \pi(x_t)) | x_0 = x}$ be value of policy $\pi$, $J_*$ be the solution of LHS, and $\pi_J(x) = \argmin_{a\in \cA} (\ell(x,a) + \gamma P_{(x,a),:} J )$ be the greedy policy with respect to $J$. Let $\nu\in \Delta_\cX$ be a probability distribution and define $\mu_{\pi, \nu} = (1-\gamma) \nu^\top (I - \gamma P^\pi)^{-1}$. 
\citet{DeFarias-VanRoy-2003} prove that for any $J$ satisfying the constraints of the LHS of \eqref{eq:ALP-primal-1}, 
\beq
\label{eq:value-greedy-policy}
\norm{J_{\pi_J} - J_*}_{1,\nu} \le \frac{1}{1-\gamma} \norm{J - J_*}_{1, \mu_{\pi_J, \nu}} \; .
\eeq
Define $\beta_u = \gamma \max_{x,a} \sum_{x'} P_{(x,a), x'} u(x')/u(x)$. Let $U=\{u\in\Real^X : u\ge \mathbf{1}, u\in \mbox{span}(\Psi), \beta_u < 1 \}$. 
%Define operator $H$ by $(H V)(x) = \max_a \sum_{y} P_{x,y}^a V(y)$ and also $\beta_V = \max_x \gamma (H V)(x)/V(x)$. 
Let $w_*$ be the solution of ALP. \citet{DeFarias-VanRoy-2003} show that for any $u\in U$,
\beq
\label{eq:ALP-apprx}
\norm{J_* - \Psi w_*}_{1,c} \le \frac{2 c^\top u}{1-\beta_{u}} \min_{w} \norm{J_* - \Psi w}_{\infty, 1/u} \; .
\eeq
This result has a number of limitations. First, solving ALP can be computationally expensive as the number of constraints is large. Second, it assumes that the feasible set of ALP is non-empty.  
%Third, it is required that $\beta_{u} < 1$. 
Finally, Inequality~\eqref{eq:value-greedy-policy} implies that $c=\mu_{\pi_{\Psi w_*}, \nu}$ is an appropriate choice to obtain performance bounds. However, $w_*$ itself is function of $c$ and is not known before solving ALP. %One approach is to solve ALP iteratively, using $c=\mu_{\pi_{\Psi w_*}, \nu}$ from last iteration. 

\citet{DeFarias-VanRoy-2004} propose a computationally efficient algorithm that is based on a constraint sampling technique. The idea is to sample a relatively small number of constraints and solve the resulting LP. 
Let $\cN\subset \Real^d$ be a known set that contains $w_*$ (solution of ALP). Let $\mu_{\pi, c}^V (x) = \mu_{\pi, c}(x) V(x)/(\mu_{\pi, c}^\top V)$ and define the distribution $\rho_{\pi, c}^V(x,a) = \mu_{\pi, c}^V (x) / A$. Let $\delta\in (0,1)$ and $\epsilon\in (0,1)$. Let $\overline\beta_u = \gamma \max_{x} \sum_{x'} P_{(x,\pi_*(x)), x'} u(x')/u(x)$ and 
\[
D = \frac{(1+\overline\beta_V) \mu_{\pi_*, c}^\top V}{2 c^\top J_*} \sup_{w\in \cN} \norm{J_* - \Psi w}_{\infty, 1/V}\,, \qquad m \ge \frac{16 A D}{(1-\gamma)\epsilon} \left( d \log \frac{48 A D}{(1-\gamma) \epsilon} + \log\frac{2}{\delta} \right) \; .
\]
Let $\cS$ be a set of $m$ random state-action pairs sampled under $\rho_{\pi_*, c}^V$. Let $\widehat w$ be a solution of the following sampled LP:
\begin{align*}
&\max_{w\in\Real^d}\  c^\top \Psi w\,,  \\
&\mbox{s.t.}\quad  w\in \cN,\, \forall (x,a)\in \cS,\, \ell(x,a) + \gamma P_{(x,a),:} \Psi w \ge (\Psi w)(x) \; .
\end{align*}
\citet{DeFarias-VanRoy-2004} prove that with probability at least $1-\delta$, we have
\[
\norm{J_* - \Psi \widehat w}_{1,c} \le \norm{J_* - \Psi  w_*}_{1,c} + \epsilon \norm{J_*}_{1,c} \; .
\]
This result has a number of limitations. First, vector $\mu_{\pi_*, c}$ (that is used in the definition of $D$) depends on the optimal policy, but an optimal policy is what we want to compute in the first place. %Second, it assumes that the ALP has a solution (see Assumption 2.1 in their paper). However, for a given set of features, there might be no feasible solution. 
Second, we can no longer use Inequality~\eqref{eq:value-greedy-policy} to obtain a performance bound (a bound on $\norm{J_{\pi_{\Psi \widehat w}} - J_*}_{1,c}$), as $\Psi \widehat w$ does not necessarily satisfy all constraints of ALP. 

\citet{Desai-Farias-Moallemi-2012} study a smoothed version of ALP, in which slack variables are introduced that allow for some violation of the constraints. Let $D'$ be a violation budget. The smoothed ALP (SALP) has the form of
\begin{align*}
&\max_{w,s} c^\top \Psi w\,, &\max_{w,s}&\, c^\top \Psi w - \frac{2 \mu_{\pi_*, c}^\top s}{1-\gamma}\,,  \\
&\mbox{s.t.}\quad \Psi w \le L \Psi w + s,\, \mu_{\pi_*, c}^\top s\le D',\, s\ge {\mathbf 0},\, &\mbox{s.t.}\quad & \Psi w \le L \Psi w + s,\, s\ge {\mathbf 0} \; .
\end{align*}
The ALP on RHS is equivalent to LHS with a specific choice of $D'$. Let $\overline{U}=\{u\in\Real^X\ : \ u\ge \mathbf{1} \}$ be a set of weight vectors. \citet{Desai-Farias-Moallemi-2012} prove that if $w_*$ is a solution to above problem, then
\[
\norm{J_* - \Psi w_*}_{1,c} \le \inf_{w, u\in \overline{U}} \norm{J_* - \Psi w}_{\infty, 1/u} \left( c^\top u + \frac{2(\mu_{\pi_*, c}^\top u) (1+ \beta_u)}{1-\gamma} \right) \; . 
\]
The above bound improves \eqref{eq:ALP-apprx} as $\overline{U}$ is larger than $U$ and RHS in the above bound is smaller than RHS of \eqref{eq:ALP-apprx}. Further, they prove that if $\eta$ is a distribution and we choose $c=(1-\gamma) \eta^\top (I-\gamma P^{\pi_{\Psi w_*}})$, then
\[
\norm{J_{\mu_{\Psi w_*}}-J_*}_{1,\eta} \le  \frac{1}{1-\gamma} \left( \inf_{w, u\in \overline{U}} \norm{J_* - \Psi w}_{\infty, 1/u} \left( c^\top u + \frac{2(\mu_{\pi_*, \nu}^\top u) (1+\beta_u)}{1-\gamma} \right)  \right) \; .
\]
Similar methods are also proposed by \citet{Petrik-Zilberstein-2009}. 
One problem with this result is that $c$ is defined in terms of $w_*$, which itself depends on $c$. %This is similar to the issue in \citep{DeFarias-VanRoy-2003}. 
Also, the smoothed ALP formulation uses $\pi_*$ which is not known. \citet{Desai-Farias-Moallemi-2012} also propose a computationally efficient algorithm. Let $\cS$ be a set of $S$ random states drawn under distribution $\mu_{\pi_*, c}$. Let $\cN'\subset \Real^d$ be a known set that contains the solution of SALP. The algorithm solves the following LP:
\begin{align*}
&\max_{w,s}\, c^\top \Psi w - \frac{2}{(1-\gamma) S}  \sum_{x\in \cS} s(x)\,,  \\
&\mbox{s.t.}\quad \forall x\in \cS,\, (\Psi w)(x) \le (L \Psi w)(x) + s(x),\, s\ge \mathbf 0,\, w\in \cN' \; .
\end{align*}
Let $\widehat w$ be the solution of this problem. 
\citet{Desai-Farias-Moallemi-2012} prove high probability bounds on the approximation error $\norm{J_* - \Psi \widehat w}_{1,c}$. However, it is no longer clear if a performance bound on $\norm{J_* - J_{\pi_{\Psi \widehat w}}}_{1,c}$ can be obtained from this approximation bound. 

%show that if we were able to sample from the stationary distribution of the optimal policy, then LP~\eqref{eq:primal-apprx} can be solved efficiently. The constant vector is in the span of basis functions. 

Next, we turn our attention to average cost ALP, which is a more challenging problem and is also the focus of this paper. Let $\nu$ be a distribution over states, $u:\cX\ra [1,\infty)$, $\eta>0$, $\gamma\in [0,1]$, $P_\gamma^\pi = \gamma P^\pi + (1-\gamma) \mathbf{1} \nu^\top$, and $L_\gamma h = \min_{\pi} (\ell_\pi + P_\gamma^\pi h)$. 
\citet{DeFarias-VanRoy-2006} propose the following optimization problem:
\begin{align}
\label{eq:primal-approx1}
&\min_{w, s_1, s_2} s_1 + \eta s_2\,, \\
\notag
&\mbox{s.t.}\quad L_\gamma \Psi w - \Psi w + s_1 \mathbf{1} + s_2 u \ge \mathbf 0,\, s_2 \ge 0 \; .
\end{align}
Let $(w_*, s_{1,*}, s_{2,*})$ be the solution of this problem. 
Define the mixing time of policy $\pi$ by 
\[
\tau_\pi = \inf\left\{ \tau \ : \ \abs{\frac{1}{t} \sum_{t'=0}^{t-1} \nu^\top (P^\pi)^{t'} \ell_\pi - \lambda_\pi   } \le \frac{\tau}{t},\, \forall t \right\} \; .
\]
Let $\tau_* = \liminf_{\delta\ra 0} \{ \tau_\pi : \lambda_\pi \le \lambda_* + \delta \}$. Let $\pi_{\gamma}^*$ be the optimal policy when discount factor is $\gamma$. Let $\pi_{\gamma, w}$ be the greedy policy with respect to $\Psi w$ when discount factor is $\gamma$, $\mu_{\gamma, \pi}^\top = 	(1-\gamma) \sum_{t=0}^\infty \gamma^t \nu^\top (P^\pi)^t$ and $\mu_{\gamma, w} = \mu_{\gamma, \pi_{\gamma, w}}$. 
 \citet{DeFarias-VanRoy-2006} prove that if $\eta \ge (2-\gamma) \mu_{\gamma, \pi_{\gamma}^*}^\top u$,
\[
\lambda_{w_*} - \lambda_* \le \frac{(1+\beta) \eta \max(D'', 1)}{1-\gamma} \min_{w} \norm{h_\gamma^* - \Psi w}_{\infty, 1/u} + (1-\gamma) (\tau_* + \tau_{\pi_{w_*}}) \,,
\]
where $\beta = \max_\pi \norm{I - \gamma P^\pi}_{\infty, 1/u}$, $D'' = \mu_{\gamma, w_*}^\top V /  (\nu^\top V)$ and $V=L_\gamma \Psi w_* - \Psi w_* + s_{1,*} \mathbf{1} +  s_{2,*} u$. Similar results are obtained more recently by \citet{Veatch-2013}. 

An appropriate choice for vector $\nu$ is $\nu=\mu_{\gamma, w_*}$. Unfortunately, $w_*$ depends on $\nu$. We should also note that solving \eqref{eq:primal-approx1} can be computationally expensive. \citet{DeFarias-VanRoy-2006} propose constraint sampling techniques similar to \citep{DeFarias-VanRoy-2004}, but no performance bounds are provided. 

\citet{Wang-Lizotte-Bowling-Schuurmans-2008} study ALP \eqref{eq:dual-apprx} and show that there is a dual form for standard value function based algorithms, including on-policy and off-policy updating and policy improvement. They also study the convergence of these methods, but no performance bounds are shown.

\section{A Reduction to Stochastic Convex Optimization}
\label{sec:reduction}

In this section, we describe our algorithm as a  reduction from Markov decision problems to stochastic convex optimization. The main idea is to convert the ALP~\eqref{eq:dual-apprx} into an unconstrained optimization over $\Theta$ by adding a function of the constraint violations to the objective, then run stochastic gradient descent with unbiased estimated of the gradient.
%%%%%%%%%%%%%%%%%%%%%%%%%%%%%%%%%%%%%%%%%%%%%%%%%%%%%%
%At a very high level, the algorithm is a stochastic subgradient descent method whose steps are guided by the loss function and also randomly chosen constraints. This leads to obtaining small losses while respecting the constraints. 
%%%%%%%%%%%%%%%%%%%%%%%%%%%%%%%%%%%%%%%%%%%%%%%%%%%%%%

For a positive constant $H$, form the following convex cost function by adding a multiple of the total constraint violations to the objective of the LP~\eqref{eq:dual-apprx}:
\begin{equation}\label{eq:objective_function}
\begin{split}	
c(\theta) &= \ell^\top (\mu_0 + \Phi \theta) + H \norm{[\mu_0 + \Phi\theta]_{-}}_1 + H \norm{(P - B)^\top (\mu_0 + \Phi\theta)}_1 \\
&= \ell^\top (\mu_0 + \Phi \theta) + H \norm{[\mu_0 + \Phi\theta]_{-}}_1+ H \norm{(P - B)^\top \Phi\theta}_1 \\
&= \ell^\top (\mu_0 + \Phi \theta) + H \sum_{(x,a)} \abs{ [\mu_0(x,a) + \Phi_{(x,a),:}\theta]_{-}} + H \sum_{x'} \abs{(P - B)_{:,x'}^\top \Phi \theta} \; .
\end{split}
\end{equation}
We justify using this surrogate loss as follows. Suppose we find a near optimal vector $\widehat \theta$ such that $c(\widehat \theta) \le \min_{\theta\in\Theta} c(\theta) + O(\epsilon)$. We will prove
\begin{enumerate}
\item that $\norm{[\mu_0 + \Phi\widehat\theta]_{-}}_1$ and $\norm{(P -
B)^\top (\mu_0 + \Phi\widehat\theta)}_1$ are small and $\mu_0 + \Phi\widehat\theta$ is close to $\mu_{\widehat\theta}$
%$\mu_{\mu_0 + \Phi\widehat\theta}$ 
(by Lemma~\ref{lem:lem1} in Section~\ref{sec:reduction}), and
\item that $\ell^\top (\mu_0 + \Phi \widehat\theta) \le \min_{\theta\in\Theta} c(\theta) + O(\epsilon)$. 
\end{enumerate}
As we will show, these two facts imply that with high probability, for any $\theta\in\Theta$,
\begin{align*}
\mu_{\widehat\theta}^\top \ell \le \mu_\theta^\top \ell &+ \frac{1}{\epsilon}  \norm{[\mu_0 + \Phi\theta]_{-}}_1+ \frac{1}{\epsilon} \norm{(P - B)^\top (\mu_0 + \Phi\theta)}_1 + O(\epsilon) \; ,
\end{align*}
which is to say that minimization of $c(\theta)$ solves the extended efficient large-scale ALP problem.

Unfortunately, calculating the gradients of $c(\theta)$ is $O(XA)$. Instead, we construct unbiased estimators and use stochastic gradient descent. Let $T$ be the number of iterations of our algorithm. Let $q_1$ and $q_2$ be distributions over the state-action and state space, respectively (we will later discuss how to choose them). Let $((x_t,a_t))_{t=1\dots T}$ be i.i.d. samples from $q_1$ and $(x_t')_{t=1\dots T}$ be i.i.d. samples from $q_2$. At round $t$, the algorithm estimates subgradient $\nabla c(\theta)$ by
\begin{align}
\label{eq:grad-est}
g_t(\theta) &= \ell^\top \Phi - H \frac{\Phi_{(x_t,a_t),:}}{q_1(x_t,a_t)} \one{\mu_0(x_t,a_t)+\Phi_{(x_t,a_t),:} \theta < 0} + H  \frac{(P-B)_{:,x_t'}^\top \Phi}{q_2(x_t')} s((P-B)_{:,x_t'}^\top \Phi \theta).
\end{align}
This estimate is fed to the projected subgradient method, which in turn generates a vector $\theta_t$. After $T$ rounds, we average vectors $(\theta_t)_{t=1\dots T}$ and obtain the final solution $\widehat \theta_T = \sum_{t=1}^T \theta_t/T$. Vector $\mu_0 + \Phi \widehat \theta_T$ defines a policy, which in turn defines a stationary distribution $\mu_{\widehat \theta_T}$.\footnote{Recall that $\mu_{\theta}$ is the stationary distribution of policy 
\[
\pi_{\theta}(a|x) = \frac{[\mu_0(x,a) + \Phi_{(x,a),:} \theta]_{+}}{\sum_{a'} [\mu_0(x,a') + \Phi_{(x,a'),:} \theta]_{+}} \; .
\]
With an abuse of notation, we use $\mu_{\theta}$ to denote policy $\pi_\theta$ as well.
} 
The algorithm is shown in Figure~\ref{alg:SGD}.

\begin{figure}
\begin{center}
\framebox{\parbox{12cm}{
\begin{algorithmic}
\STATE \textbf{Input: } Constant $S>0$, number of rounds $T$, constant $H$.  %stationary distribution $\mu_0$, distributions $q_1$ and $q_2$, 
\STATE Let $\Pi_{\Theta}$ be the Euclidean projection onto $\Theta$. 
\STATE Initialize $\theta_1 = 0$. 
\FOR{$t:=1,2,\dots, T$}
\STATE Sample $(x_t,a_t)\sim q_1$ and $x_t'\sim q_2$. 
\STATE Compute subgradient estimate $g_t$ \eqref{eq:grad-est}.
\STATE Update $\theta_{t+1} = \Pi_{\Theta} (\theta_t - \eta_t g_t)$. 
\ENDFOR
\STATE $\widehat \theta_T = \frac{1}{T}\sum_{t=1}^T \theta_t$.
\STATE Return policy $\pi_{\widehat \theta_T}$.
\end{algorithmic}
}}
\end{center}
\caption{The Stochastic Subgradient Method for Markov Decision Processes}
\label{alg:SGD}
\end{figure}

\subsection{Analysis}
In this section, we state and prove our main result, Theorem~\ref{thm:main}. We begin with a discussion of the assumptions we make then follow with the main theorem. We break the proof into two main ingredients. First, we demonstrate that a good approximation to the surrogate loss gives a feature vector that is almost a stationary distribution; this is Lemma \ref{lem:lem1}.
Second, we justify the use of unbiased gradients in Theorem \ref{thm:stoch-gradient} and Lemma \ref{lem:risk-bound}. The section concludes with the proof.

We make a mixing assumption on the MDP so that any policy quickly converges to its stationary distribution. 
\begin{ass}[Fast Mixing]
\label{ass:uniform-mixing}
For any policy $\pi$, there exists a constant $\tau(\pi)>0$ such that for all distributions $d$ and $d'$ over the state space, $\norm{d P^\pi - d' P^\pi}_1 \le e^{-1/\tau(\pi)} \norm{d - d'}_1$. 
\end{ass}
%Finally, we need the following assumption to ensure that cost estimates are bounded. 
Define
\begin{align*}
C_1 = \max_{(x,a)\in \cX\times \cA}\frac{\norm{\Phi_{(x,a),:}}}{q_1(x, a)}\,, \qquad C_2 = \max_{x\in \cX}\frac{\norm{(P - B)_{:,x}^\top \Phi}}{q_2(x)} \; .
\end{align*}
%Constants $C_1$ and $C_2$ do not depend on the size of the state space. 
These constants appear in our performance bounds. So we would like to choose distributions $q_1$ and $q_2$ such that $C_1$ and $C_2$ are small. 
For example, if there is $C'>0$ such that for any $(x,a)$ and $i$, $\Phi_{(x,a), i}\le C'/(X A)$ and each column of $P$ has only $N$ non-zero elements, then we can simply choose $q_1$ and $q_2$ to be  uniform distributions. Then it is easy to see that 
\begin{align*}
\frac{\norm{\Phi_{(x,a),:}}}{q_1(x, a)} \le C'\,, \qquad \frac{\norm{(P - B)_{:,x}^\top \Phi}}{q_2(x)} \le C' (N + A) \; .
\end{align*}
%\todoy{I guess the above bounds can be improved}
%is always bounded and thus, $\EE{e^{a (c_t(\theta)-c(\theta))^2}} \le 2$ for an appropriate $a=O(1/H^2)$. This in turn implies that $c_t(\theta)-c(\theta)$ is a $\sqrt{3/a}$-sub-Gaussian random variable. \todoy{ref} 
As another example, if $\Phi_{:, i}$ are exponential distributions and feature values at neighboring states are close to each other, then we can choose $q_1$ and $q_2$ to be appropriate exponential distributions so that $\norm{\Phi_{(x,a),:}}/q_1(x, a)$ and $\norm{(P - B)_{:,x}^\top \Phi}/q_2(x)$ are always bounded. Another example is when there exists a constant $C''>0$ such that for any $x$, $\norm{P_{:,x}^\top \Phi}/\norm{B_{:,x}^\top \Phi} < C''$\footnote{This condition requires that columns of $\Phi$ are close to their \textit{one step look-ahead}.}  and we have access to an efficient algorithm that computes $Z_1 = \sum_{(x,a)} \norm{\Phi_{(x,a),:}}$ and $Z_2 = \sum_x \norm{B_{:,x}^\top \Phi}$ and can sample from $q_1(x,a)= \norm{\Phi_{(x,a),:}}/Z_1$ and $q_2(x)=\norm{B_{:,x}^\top \Phi}/Z_2$. In what follows, we assume that such distributions $q_1$ and $q_2$ are known.

We now state the main theorem.
\begin{thm}
\label{thm:main}
Consider an expanded efficient large-scale dual ALP problem. %Suppose we have access to a backward simulator and $C1/C2$-distributions. 
Suppose we apply the stochastic subgradient method (shown in Figure~\ref{alg:SGD}) to the problem. Let $\epsilon\in (0,1)$. Let $T=1/\epsilon^4$ be the number of rounds and $H=1/\epsilon$ be the constraints multiplier in subgradient estimate \eqref{eq:grad-est}. Let $\widehat \theta_T$ be the output of the stochastic subgradient method after $T$ rounds and let the learning rate be $\eta_1=\dots=\eta_T=S/(G'\sqrt{T})$, where $G' = \sqrt{d} + H (C_1 + C_2)$. Define $V_1(\theta) = \sum_{(x,a)} \abs{ [\mu_0(x,a) + \Phi_{(x,a),:}\theta]_{-}}$ and $V_2(\theta) = \sum_{x'} \abs{(P - B)_{:,x'}^\top ( \mu_0 + \Phi \theta) }$. 
Then, for any $\delta\in (0,1)$, with probability at least $1-\delta$,
\beq
\label{eq:exp-loss}
\mu_{\widehat\theta_T}^\top \ell \le \min_{\theta\in\Theta} \left( \mu_\theta^\top \ell + \frac{1}{\epsilon} ( V_1(\theta) + V_2(\theta) ) + O(\epsilon) \right) \; , %O(\epsilon \log(1/\delta)) \; .
\eeq
where constants hidden in the big-O notation are polynomials in $S$, $d$, $C_1$, $C_2$, $\log(1/\delta)$, $V_1(\theta)$, $V_2(\theta)$, $\tau(\mu_{\theta})$, and $\tau(\mu_{\widehat\theta_T})$. 
\end{thm}
Functions $V_1$ and $V_2$ are bounded by small constants for any set of normalized features: for any $\theta\in\Theta$,
\begin{align*}
V_1(\theta) &\le \norm{\mu_0}_1 + \norm{\Phi \theta}_1 \le 1 + \sum_{(x,a)} \abs{\Phi_{(x,a),:} \theta} \le 1 + S d \,, \\
V_2(\theta) &\le \sum_{x'} \abs{P_{:,x'}^\top ( \mu_0 + \Phi \theta) } + \sum_{x'} \abs{B_{:,x'}^\top ( \mu_0 + \Phi \theta) } \\ 
&\le \left(\sum_{x'} P_{:,x'}\right)^\top [ \mu_0 + \Phi \theta ]_+ + \left(\sum_{x'} B_{:,x'}\right)^\top [\mu_0 + \Phi \theta]_+ \\  
&= (2) \mathbf 1^\top [ \mu_0 + \Phi \theta ]_+ \\
&\le (2) \mathbf 1^\top \abs{ \mu_0 + \Phi \theta } \\
&= 2 + 2 S \; .
\end{align*}
Thus $V_1$ and $V_2$ can be very small given a carefully designed set of features. The output $\widehat \theta_T$ is a random vector as the algorithm is based on a stochastic convex optimization method. The above theorem shows that with high probability the policy implied by this output is near optimal. 

The optimal choice for $\epsilon$ is $\epsilon=\sqrt{V_1(\theta_*) + V_2(\theta_*)}$, where $\theta_*$ is the minimizer of RHS of \eqref{eq:exp-loss} and not known in advance. Once we obtain $\widehat \theta_T$, we can estimate $V_1(\widehat \theta_T)$ and $V_2(\widehat \theta_T)$ and use input $\epsilon=\sqrt{V_1(\widehat \theta_T) + V_2(\widehat \theta_T)}$ in a second run of the algorithm. This implies that the error bound scales like $O(\sqrt{V_1(\theta_*) + V_2(\theta_*)})$.

The next lemma, providing the first ingredient of the proof, relates the amount of constraint violation of a vector $\theta$ to resulting stationary distribution $\mu_\theta$. %The proof can be found in Appendix~\ref{app:proof1}.
\begin{lem}
\label{lem:lem1}
Let $u\in \Real^{X A}$ be a vector. 
Let $\cN$ be the set of points $(x,a)$ where $u (x,a)<0$ and $\cS$ be complement of $\cN$.  
Assume
\[
\sum_{x,a} u (x,a) = 1, \sum_{(x,a)\in \cN } \abs{u(x,a)} \le \epsilon', \norm{u^\top (P-B)}_1\le \epsilon'' .
\]
Vector $[u]_+/\norm{[u]_+}_1$ defines a policy, which in turn defines a stationary distribution $ \mu_{u}$. We have that 
\[
\norm{ \mu_{u} - u}_1 \le \tau(\mu_u) \log(1/\epsilon') (2\epsilon'+\epsilon'') + 3\epsilon' \; .
\]
\end{lem}
\begin{proof}%[Proof of Lemma~\ref{lem:lem1}]
Let $f = u^\top (P-B)$. From $\norm{u^\top (P-B)}_1 \le \epsilon''$, we get that for any $x'\in \cX$,
\begin{align*}
\sum_{(x,a)\in \cS} & u(x,a) (P-B)_{(x,a), x'} = - \sum_{(x,a)\in \cN} u(x,a) (P-B)_{(x,a), x'} + f(x')
\end{align*}
such that $\sum_{x'} \abs{f(x')} \le \epsilon''$. Let $h = [u]_+/\norm{[u]_+}_1$. Let $H' = \norm{h^\top (B-P)}_1$. We write
\begin{align*}
H' &= \sum_{x'} \abs{\sum_{(x,a)\in \cS} h(x,a)  (B-P)_{(x,a),x'} } \\ 
&=  \frac{1}{1+\epsilon'} \sum_{x'} \abs{\sum_{(x,a)\in \cS} u(x,a)  (B-P)_{(x,a),x'} } \\ 
&= \frac{1}{1+\epsilon'}  \sum_{x'} \abs{-\sum_{(x,a)\in \cN} u(x,a) (B-P)_{(x,a),x'} + f(x')} \\
&\le \frac{1}{1+\epsilon'} \left( \sum_{x'} \abs{-\sum_{(x,a)\in \cN} u(x,a) (B-P)_{(x,a),x'}} + \sum_{x'} \abs{f(x')} \right) \\
&\le \frac{1}{1+\epsilon'} \left( \epsilon'' + \sum_{(x,a)\in \cN} \sum_{x'} \abs{u(x,a)} \abs{(B-P)_{(x,a),x'}} \right)  \\
&\le \frac{1}{1+\epsilon'} \left( \epsilon'' + \sum_{(x,a)\in \cN} 2 \abs{u(x,a)} \right) \le \frac{2\epsilon' + \epsilon''}{1+\epsilon'} \\ 
&\le 2 \epsilon' + \epsilon''  \; .
\end{align*}
Vector $h$ is almost a stationary distribution in the sense that
\beq
\label{eq:almost-stationary}
\norm{h^\top (B-P)}_1 \le 2 \epsilon' + \epsilon''\; .
\eeq
Let $\norm{w}_{1,\cS} = \sum_{(x,a)\in \cS} \abs{w(x,a)}$. First, we have that
\begin{align*}
\norm{h - u}_1 &\le \norm{h - \frac{u}{1+\epsilon'}}_1 + \norm{u - \frac{u}{1+\epsilon'}}_{1,\cS} \le 2 \epsilon' \; .
\end{align*}
Next we bound $\norm{ \mu_{h} - h}_1$. Let $\nu_0 = h$ be the initial state distribution. We will show that as we run policy $h$ (equivalently, policy $ \mu_{h}$), the state distribution converges to $ \mu_{h}$ and this vector is close to $h$. From \eqref{eq:almost-stationary}, we have $\mu_0^\top P = h^\top B + v_0$, where $v_0$ is such that $\norm{v_0}_1\le 2\epsilon'+\epsilon''$. Let $M^{h}$ be a $X\times (XA)$ matrix that encodes policy $h$, $M_{(i,(i-1)A+1)\mbox{-}(i,iA)}^{h}=h(\cdot|x_i)$. Other entries of this matrix are zero. We get that
\begin{align*}
h^\top P M^{h} &= (h^\top B + v_0) M^{h} = h^\top B M^{h} + v_0 M^{h}= h^\top + v_0 M^{h} \,,
 \end{align*}
where we used the fact that $h^\top B M^{h} = h^\top$. Let $\mu_1^\top = h^\top P M^{h}$ which is the state-action distribution after running policy $h$ for one step. Let $v_1 = v_0 M^{h} P = v_0 P^{h}$ and notice that as $\norm{v_0}_1\le 2\epsilon'+\epsilon''$, we also have that $\norm{v_1}_1 = \norm{P^{h \top} v_0^\top}_1 \le \norm{v_0}_1\le 2\epsilon'+\epsilon''$. Thus,
\[
\mu_1^\top P = h^\top P + v_1 =  h^\top B + v_0 + v_1  \; .
\]
By repeating this argument for $k$ rounds, we get that
\[
\mu_k^\top = h^\top + (v_0 + v_1 + \dots + v_{k-1}) M^{h}
\]
and it is easy to see that 
\[\norm{(v_0 + v_1 + \dots + v_{k-1}) M^{h}}_1 \le \sum_{i=0}^{k-1} \norm{v_i}_1 \le k( 2\epsilon'+\epsilon'').
\] Thus, $\norm{\mu_k - h}_1 \le k (2 \epsilon'+\epsilon'')$. Now notice that $\mu_k$ is the state-action distribution after $k$ rounds of policy $ \mu_{h}$. Thus, by mixing assumption, $\norm{\mu_k -  \mu_{h}}_1 \le e^{-k/\tau(h)}$. By the choice of $k=\tau(h) \log (1/\epsilon')$, we get that $\norm{ \mu_{h} - h}_1 \le \tau(h) \log(1/\epsilon') (2\epsilon'+\epsilon'') + \epsilon'$. 

\end{proof}

The second ingredient is the validity of using estimates of the subgradients. We assume access to estimates of the subgradient of a convex cost function. Error bounds can be obtained from results in the stochastic convex optimization literature; the following theorem, a high-probability version of Lemma~3.1 of \citet{Flaxman-Kalai-McMahan-2005} for stochastic convex optimization, is sufficient. 
\begin{thm}
\label{thm:stoch-gradient}
Let $Z$ be a positive constant and $\cZ$ be a bounded subset of $\Real^d$ such that for any $z\in \cZ$, $\norm{z}\le Z$. 
Let $(f_t)_{t=1,2,\dots,T}$ be a sequence of real-valued convex cost functions defined over $\cZ$. Let $z_1,z_2,\dots,z_T\in \cZ$ be defined by $z_1=0$  and $z_{t+1} = \Pi_{\cZ} (z_t - \eta f_t')$, where $\Pi_{\cZ}$ is the Euclidean projection onto $\cZ$, $\eta>0$ is a learning rate, and $f_1',\dots,f_T'$ are unbiased subgradient estimates such that $\EE{f_t' | z_t} = \nabla f(z_t)$ and $\norm{f_t'}\le F$ for some $F>0$. Then, for $\eta = Z/(F\sqrt{T})$, for any $\delta\in (0,1)$, with probability at least $1-\delta$,
\begin{align}
\label{eq:stoch-gradient-regret}
\sum_{t=1}^T &f_t(z_t) - \min_{z\in\cZ} \sum_{t=1}^T f_t(z) \le Z F \sqrt{T} + \sqrt{(1 + 4 Z^2 T ) \left(2 \log\frac{1}{\delta} + d \log \left( 1 + \frac{Z^2 T}{d} \right) \right) } \; .
\end{align}
\end{thm}
\begin{proof}
Let $z_* = \argmin_{z\in\cZ} \sum_{t=1}^T f_t(z)$ and $\eta_t = f_t' - \nabla f_t (z_t)$. Define function $h_t:\cZ\ra\Real$ by $h_t(z) = f_t(z) + z \eta_t$. Notice that $\nabla h_t(z_t) = \nabla f_t(z_t) + \eta_t = f_t'$. By Theorem~1 of \citet{Zinkevich-2003}, we get that
\[
\sum_{t=1}^T h_t(z_t) - \sum_{t=1}^T h_t(z_*) \le \sum_{t=1}^T h_t(z_t) - \min_{z\in\cZ}\sum_{t=1}^T h_t(z) \le Z F \sqrt{T} \; .
\] 
Thus,
\[
\sum_{t=1}^T f_t(z_t) - \sum_{t=1}^T f_t(z_*) \le Z F \sqrt{T} + \sum_{t=1}^T (z_* - z_t) \eta_t \; .
\]
Let $S_t = \sum_{s=1}^{t-1} (z_* - z_s) \eta_s$, which is a self-normalized sum~\citep{delaPena-Lai-Shao-2009}. By Corollary~3.8 and Lemma~E.3 of \citet{Abbasi-Yadkori-2012}, we get that for any $\delta\in (0,1)$, with probability at least $1-\delta$,
\begin{align*}
\abs{S_t} &\le \sqrt{\left(1 + \sum_{s=1}^{t-1} (z_t - z_*)^2 \right) \left(2 \log\frac{1}{\delta} + d \log \left( 1 + \frac{Z^2 t}{d} \right) \right) }  \\
&\le \sqrt{(1 + 4 Z^2 t ) \left(2 \log\frac{1}{\delta} + d \log \left( 1 + \frac{Z^2 t}{d} \right) \right) } \; .
\end{align*}
Thus,
\[
\sum_{t=1}^T f_t(z_t) - \min_{z\in\cZ} \sum_{t=1}^T f_t(z) \le Z F \sqrt{T} + \sqrt{(1 + 4 Z^2 T ) \left(2 \log\frac{1}{\delta} + d \log \left( 1 + \frac{Z^2 T}{d} \right) \right) } \; .
\]
\end{proof}

\begin{rem}
\label{rem:Jensen}
Let $B_T$ denote RHS of \eqref{eq:stoch-gradient-regret}. 
If all cost functions are equal to $f$, then by convexity of $f$ and an application of Jensen's inequality, we obtain that $f(\sum_{t=1}^T z_t/T) - \min_{z\in \cZ} f(z) \le B_T/T$. 
\end{rem}
As the next lemma shows, Theorem~\ref{thm:stoch-gradient} can be applied in our problem to optimize cost function $c$. 
\begin{lem}
\label{lem:risk-bound}
Under the same conditions as in Theorem~\ref{thm:main}, we have that for any $\delta\in (0,1)$, with probability at least $1-\delta$, 
\begin{align}
\label{eq:err-bnd}
c(\widehat \theta_T) &-  \min_{\theta\in\Theta} c(\theta) \le \frac{S G'}{\sqrt{T}} + \sqrt{\frac{1 + 4 S^2 T}{T^2} \left(2 \log\frac{1}{\delta} + d \log \left( 1 + \frac{S^2 T}{d} \right) \right) } \; .
\end{align}
\end{lem}
\begin{proof}
We prove the lemma by showing that conditions of Theorem~\ref{thm:stoch-gradient} are satisfied. We begin by calculating the subgradient and bounding its norm with a quantity independent of the number of states. 
If $\mu_0(x,a)+\Phi_{(x,a),:} \theta \ge 0$, then $\nabla_{\theta} \abs{[\mu_0(x,a)+\Phi_{(x,a),:} \theta]_{-}} = 0$. Otherwise, $\nabla_{\theta} \abs{[\mu_0(x,a)+\Phi_{(x,a),:} \theta]_{-}} = -\Phi_{(x,a),:}$. 
Calculating,
\begin{equation}\label{eq:cost_gradient}
\begin{split}
\nabla_{\theta} c(\theta) &= \ell^\top \Phi + H \sum_{(x,a)} \nabla_{\theta} \abs{[\mu_0(x,a)+\Phi_{(x,a),:} \theta]_{-}}+ H \sum_{x'}\nabla_{\theta} \abs{ (P-B)_{:,x'}^\top \Phi \theta}  \\
&= \ell^\top \Phi - H \sum_{(x,a)} \Phi_{(x,a),:} \one{\mu_0(x,a)+\Phi_{(x,a),:} \theta < 0}+ H \sum_{x'}  (P-B)_{:,x'}^\top \Phi s((P-B)_{:,x'}^\top \Phi \theta)   \,,
\end{split}
\end{equation}
where $s(z) = \one{z>0} - \one{z<0} $ is the sign function. Let $\pm$ denote the plus or minus sign (the exact sign does not matter here). 
Let $G=\norm{\nabla_{\theta} c(\theta)}$. We have that
\begin{align*}
G &\le H \sqrt{\sum_{i=1}^d \left(\sum_{x'} \left(\pm\sum_{(x,a)} (P-B)_{(x,a),x'} \Phi_{(x,a),i} \right) \right)^2 } +\norm{ \ell^\top \Phi} + H \sqrt{\sum_{i=1}^d \left(\sum_{(x,a)} \abs{\Phi_{(x,a),i}} \right)^2 } \; .
\end{align*}
Thus,
\begin{align*}
G &\le \sqrt{\sum_{i=1}^d (\ell^\top \Phi_{:,i})^2} + H \sqrt{d} + H \sqrt{\sum_{i=1}^d \left(\sum_{(x,a)} \left(\pm\sum_{x'} (P-B)_{(x,a),x'}  \right) \Phi_{(x,a),i} \right)^2 } \\
&\le \sqrt{d} + H \sqrt{d}  + H \sqrt{\sum_{i=1}^d \left(2 \sum_{(x,a)} \abs{\Phi_{(x,a),i}} \right)^2 } =\sqrt{d} (1 + 3 H )  \,,
\end{align*}
where we used $\abs{ \ell^\top \Phi_{:,i}} \le \norm{\ell}_\infty \norm{\Phi_{:,i}}_1 \le 1$.

%The terms $P_{:,x}^\top \Phi\theta$ and $B_{:,x}^\top \Phi\theta$ can be computed in $O(N+A)$ time as we assume that each column of $P$ has only $N$ non-zero elements (see Assumption~\ref{ass:sparse_connected}). 
Next, we show that norm of the subgradient estimate is bounded by $G'$:
\[
\norm{g_t} \le \norm{\ell^\top \Phi} + H \frac{\norm{\Phi_{(x_t,a_t),:}} }{q_1(x_t,a_t)} + H  \frac{\norm{(P-B)_{:,x_t'}^\top \Phi}}{q_2(x_t')} \le \sqrt{d} + H (C_1 + C_2) \; .
\]

Finally, we show that the subgradient estimate is unbiased:
\begin{align*}
\EE{g_t(\theta)} &= \ell^\top \Phi - H \sum_{(x,a)} q_1(x,a) \frac{\Phi_{(x,a),:}}{q_1(x,a)} \one{\mu_0(x,a)+\Phi_{(x,a),:} \theta < 0} \\
&\qquad\qquad\qquad\qquad+ H \sum_{x'} q_2(x') \frac{(P-B)_{:,x'}^\top \Phi}{q_2(x')} s((P-B)_{:,x'}^\top \Phi \theta) \\
&= \ell^\top \Phi - H \sum_{(x,a)} \Phi_{(x,a),:} \one{\mu_0(x,a)+\Phi_{(x,a),:} \theta < 0}+ H \sum_{x'}  (P-B)_{:,x'}^\top \Phi s((P-B)_{:,x'}^\top \Phi \theta) \\
&= \nabla_{\theta} c(\theta) \; .
\end{align*}
The result then follows from Theorem~\ref{thm:stoch-gradient} and Remark~\ref{rem:Jensen}. 

\end{proof}

With both ingredients in place, we can prove our main result.
\begin{proof}[Proof of Theorem~\ref{thm:main}]
Let $b_T$ be the RHS of \eqref{eq:err-bnd}. Equation \eqref{eq:err-bnd} implies that with high probability for any $\theta\in\Theta$,
\begin{align}
\label{eq:online-to-batch}
\ell^\top (\mu_0 + \Phi \widehat \theta_T) + H\, V_1(\widehat\theta_T) + H\, V_2(\widehat\theta_T) \le \ell^\top (\mu_0 + \Phi\theta)+ H\, V_1(\theta)  + H\, V_2(\theta) + b_T\; .
\end{align}
From \eqref{eq:online-to-batch}, we get that
\begin{align}
\label{eq:V1}
V_1(\widehat \theta_T) &\le \frac{1}{H} \left( 2(1+S) + H\, V_1(\theta)  + H\, V_2(\theta)+ b_T \right) \eqdef \epsilon' \,, \\
\label{eq:V2}
V_2(\widehat \theta_T) &\le \frac{1}{H} \left( 2(1+S) + H\, V_1(\theta)  + H\, V_2(\theta) + b_T \right) \eqdef \epsilon'' \; .
\end{align}
Inequalities \eqref{eq:V1} and \eqref{eq:V2} and Lemma~\ref{lem:lem1} give the following bound:
\beq
\label{eq:eq1}
\abs{ \mu_{\widehat\theta_T}^\top \ell - (\mu_0 + \Phi \widehat\theta_T)^\top \ell} \le \tau(\mu_{\widehat\theta_T}) \log(1/\epsilon') (2\epsilon'+\epsilon'') + 3\epsilon' \; .
\eeq
From \eqref{eq:online-to-batch} we also have
\[
\ell^\top (\mu_0 + \Phi \widehat \theta_T) \le \ell^\top (\mu_0 + \Phi\theta) + H\, V_1(\theta)  + H\, V_2(\theta) + b_T\; ,
\]
which, together with \eqref{eq:eq1} and Lemma~\ref{lem:lem1}, gives the final result:
\begin{align*}
\mu_{\widehat\theta_T}^\top \ell &\le \ell^\top (\mu_0 + \Phi\theta) + H\, V_1(\theta)  + H\, V_2(\theta) + b_T + \tau(\mu_{\widehat\theta_T}) \log(1/\epsilon') (2\epsilon'+\epsilon'') + 3\epsilon' \\
&\le \mu_{\theta}^\top \ell + H\, V_1(\theta)  + H\, V_2(\theta) + b_T+ \tau(\mu_{\widehat\theta_T}) \log(1/\epsilon') (2\epsilon'+\epsilon'') + 3\epsilon' \\ 
&\quad+ \tau(\mu_{\theta}) \log(1/V_1(\theta)) (2V_1(\theta)+V_2(\theta)) + 3V_1(\theta) \; .
\end{align*}
Recall that $b_T=O(H/\sqrt{T})$. Because $ H = 1 / \epsilon$ and $T=1/\epsilon^4$, we get that with high probability, for any $\theta\in\Theta$, $\mu_{\widehat\theta_T}^\top \ell \le \mu_\theta^\top \ell + \frac{1}{\epsilon} ( V_1(\theta) + V_2(\theta) ) + O(\epsilon)$.

\end{proof}

%%%%%%%%%%%%%%%%%%%%%%%%%%%%%%%%%%%%%%%%%%%%%%%%%%%%%%
Let's compare Theorem~\ref{thm:main} with results of \citet{DeFarias-VanRoy-2006}. Their approach is to relate the original MDP to a perturbed version\footnote{In a perturbed MDP, the state process restarts with a certain probability to a \textit{restart distribution}. Such perturbed MDPs are closely related to discounted MDPs.} and then analyze the corresponding ALP. (See Section~\ref{sec:related-work} for more details.) Let $\Psi$ be a feature matrix that is used to estimate value functions. Recall that $\lambda_*$ is the average loss of the optimal policy and $\lambda_{w}$ is the average loss of the greedy policy with respect to value function $\Psi w$. Let $h_\gamma^*$ be the differential value function when the restart probability in the perturbed MDP is $1-\gamma$. For vector $v$ and positive vector $u$, define the weighted maximum norm $\norm{v}_{\infty, u} = \max_x u(x) \abs{v(x)}$. \citet{DeFarias-VanRoy-2006} prove that for appropriate constants $C,C'>0$ and weight vector $u$,  
\beq
\label{eq:deFaVaRo}
\lambda_{w_*} - \lambda_* \le \frac{C}{1-\gamma} \min_{w} \norm{h_\gamma^* - \Psi w}_{\infty, u} + C' (1-\gamma) \; .
\eeq
%Choosing $\gamma$ close to 1 has a similar effect as choosing $\epsilon$ close to 0. 
This bound has similarities to bound \eqref{eq:exp-loss}: tightness of both bounds depends on the quality of feature vectors in representing the relevant quantities (stationary distributions in \eqref{eq:exp-loss} and value functions in \eqref{eq:deFaVaRo}). Once again, we emphasize that the algorithm proposed by \citet{DeFarias-VanRoy-2006} is computationally expensive and requires access to a distribution that depends on optimal policy.  
%%%%%%%%%%%%%%%%%%%%%%%%%%%%%%%%%%%%%%%%%%%%%%%%%%%%%%

%\documentclass{article} % For LaTeX2e
%\usepackage{todonotes}

%\usepackage{enumerate}
%\usepackage[round,comma]{natbib}
%\bibliographystyle{plainnat}
%\usepackage{amsthm}

%\input{Commands}
%\title{Stochastic gradient methods for approximate MDP policy iteration}
%\begin{document}

%\maketitle

\section{Sampling Constraints}
\label{sec:const-sampling}

In this section we describe our second algorithm for average cost MDP problems. Using the results on polytope constraint sampling \citep{DeFarias-VanRoy-2004, Calafiore-Campi-2005, Campi-Garatti-2008}, we reduce approximate the solution to the dual ALP with the solution to a smaller, sampled LP. Basically, \citet{DeFarias-VanRoy-2004} claim that given a set of affine constraints in $\Real^d$ and some measure $q$ over these constraints, if we sample $k=O(d\log(1/\delta)/\epsilon)$ constraints, then with probability at least $1-\delta$, any point that satisfies all of these $k$ sampled constraints also satisfies $1-\epsilon$ of the original set of constraints under measure $q$. This result is independent of the number of original constraints. 

Let $\mathcal L$ be a family of affine constraints indexed by $i$: constraint $i$ is satisfied at point $w\in\Real^{d}$ if $a_i^\top w+b_i\geq 0$. Let $\mathcal I$ be the family of constraints by selecting $k$ random constraints in $\mathcal L$ with respect to measure $q$.
\begin{thm}[\citet{DeFarias-VanRoy-2004}]
\label{thm:vanroy}
Assume there exists a vector that satisfies all constraints in $\mathcal L$. 
For any $\delta$ and $\epsilon$, if we take $m\geq\frac{4}{\epsilon}\left(d\log\frac{12}{\epsilon}+\log\frac{2}{\delta}\right)$, 
then, with probability $1-\delta$, a set $\mathcal I$ of $m$ i.i.d. random variables drawn from $\mathcal L$ with respect to distribution $q$ satisfies
\[
\sup_{\{w : \forall i\in\mathcal I, a_i^\top w+b_i\geq 0\}} q(\{j : a_j^\top w+b_j<0\})\leq\epsilon \;.
\]
\end{thm}

Our algorithm takes the following inputs: a positive constant $S$, a stationary distribution $\mu_0$, a
set $\Theta = \{ \theta \ : \ \theta^\top \Phi^\top \mathbf 1 = 1-\mu_0^\top \mathbf 1,\, \norm{\theta} \le S \}$, a distribution
$q_1$ over the state-action space, a distribution $q_2$ over the
state space, and constraint violation
functions $v_1:\cX\times \cA\ra [-1,0]$ and $v_2:\cX\ra [0,1]$.
%Let $v_1:\cX\times \cA\ra [-1,0]$ and $v_2:\cX\ra [0,1]$ be two functions that are inputs to our algorithm.
We will consider two families of constraints:
\begin{align*}
\mathcal L_{1} &= \{\mu_0(x,a)+\Phi_{(x,a),:}\theta\geq v_1(x,a) \mid (x,a)\in \cX\times \cA\}\,, \\
\mathcal L_2 &= \left\{ (P-B)_{:,x}^\top(\mu_0+\Phi\theta) \leq v_2(x) \mid x\in \cX \right\}\bigcup\left\{ (P-B)_{:,x}^\top(\mu_0+\Phi\theta) \geq -v_2(x) \mid x\in \cX \right\} \; .
\end{align*}
Let $\theta_*$ be the solution of 
\begin{align}
\label{eq:slp2}
&\min_{\theta\in\Theta} \ell^\top(\mu_0+\Phi\theta)\,, \\
\notag
&\mbox{s.t.}\quad \theta\in \mathcal L_1,\, \theta\in\mathcal L_2,\, \theta \in \Theta \;.
\end{align}
%Note that we are guaranteed to have $\mathbf 1^\top\Phi\theta=1$ since we constrain $\mathbf 1^\top\theta=1$. \todoy{now we have $\mu_0$} 
%Let $\mathcal I_{i}$ be $k_{i}$ constraints sampled under distribution $q_i$ for $i=1,2$. 
The constraint sampling algorithm is shown in Figure~\ref{alg:cs}. We refer to \eqref{eq:slp} as the sampled ALP, while we refer to \eqref{eq:dual-apprx} as the full ALP. 
\begin{figure}
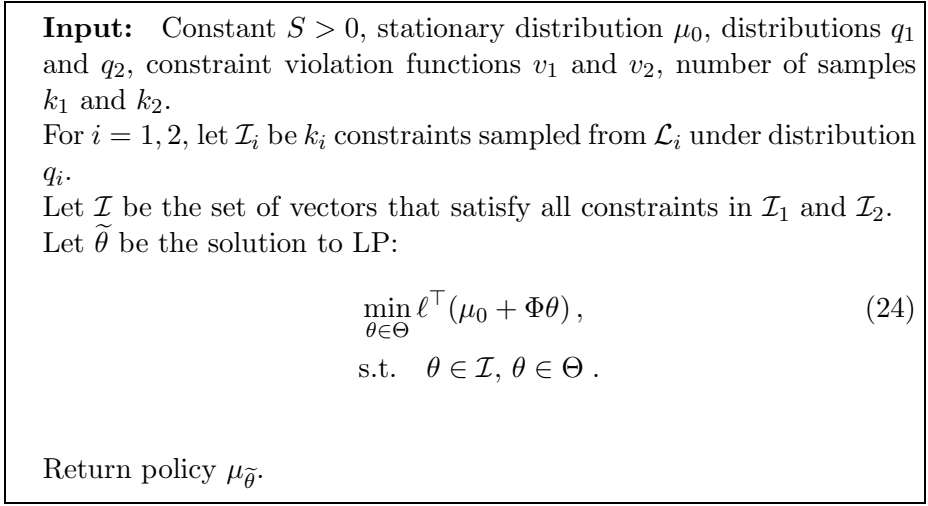

\begin{center}
\framebox{\parbox{12cm}{
\begin{algorithmic}
\STATE \textbf{Input: } Constant $S>0$, stationary distribution $\mu_0$, distributions $q_1$ and $q_2$, constraint violation functions $v_1$ and $v_2$, number of samples $k_1$ and $k_2$. 
\STATE For $i=1,2$, let $\mathcal I_{i}$ be $k_{i}$ constraints sampled from $\mathcal L_{i}$ under distribution $q_i$. 
\STATE Let  $\mathcal I$ be the set of vectors that satisfy all constraints in $\mathcal I_1$ and $\mathcal I_2$.
\STATE Let $\widetilde\theta$ be the solution to LP:
\begin{align}
\label{eq:slp}
&\min_{\theta\in\Theta} \ell^\top(\mu_0+\Phi\theta)\,, \\
\notag
&\mbox{s.t.}\quad \theta\in \mathcal I,\, \theta \in \Theta \;.
\end{align}
\STATE Return policy $\mu_{\widetilde \theta}$.
\end{algorithmic}
}}
\end{center}
\caption{The Constraint Sampling Method for Markov Decision Processes}
\label{alg:cs}
\end{figure}

\subsection{Analysis}

We require Assumption~\ref{ass:uniform-mixing} as well as:
\begin{ass}[Feasibility]
There exists a vector that satisfies all constraints $\mathcal L_1$ and $\mathcal L_2$. 
\end{ass}
Validity of this assumption depends on the choice of functions $v_1$ and $v_2$. Larger functions ensure that this assumption is satisfied, but as we show, this leads to larger error. 

The next two lemmas apply theorem \ref{thm:vanroy} to constraints $\mathcal L_1$ and $\mathcal L_2$, respectively. 
\begin{lemma}
\label{lem:sim_error}
Let $\delta_1\in (0,1)$ and $\epsilon_1\in (0,1)$. 
If we choose $k_1=\frac{4}{\epsilon_1}\left(d\log\frac{12}{\epsilon_1}+\log\frac{2}{\delta_1}\right)$, then with probability at least $1-\delta_1$, $\sum_{(x,a)} \abs{[\mu_0(x,a)+\Phi_{(x,a),:} \widetilde\theta]_{-}} \le S C_1 \epsilon_1 + \norm{v_1}_1$.
\end{lemma}
\begin{proof} 
Applying theorem \ref{thm:vanroy}, we have that w.p. $1-\delta_1$, $q_1(\mu_0(x,a)+\Phi_{(x,a),:}\widetilde\theta\geq v_1(x,a))\geq 1-\epsilon_{1}$, and thus 
\[
\sum_{(x,a)} q_1(x,a) \one{\mu_0(x,a)+\Phi_{(x,a),:} \widetilde\theta < v_1(x,a)} \le \epsilon_{1} \; .
\]
Let $L = \sum_{(x,a)} \abs{[\mu_0(x,a)+\Phi_{(x,a),:} \widetilde\theta]_{-}}$. With probability $1-\delta_1$,
\begin{align*}
L &= \sum_{(x,a)} \abs{[\mu_0(x,a)+\Phi_{(x,a),:} \widetilde\theta]_{-}} \one{\mu_0(x,a)+\Phi_{(x,a),:} \widetilde\theta \le v_1(x,a)} \\ 
&\quad+ \sum_{(x,a)} \abs{[\mu_0(x,a)+\Phi_{(x,a),:} \widetilde\theta]_{-}} \one{\mu_0(x,a)+\Phi_{(x,a),:} \widetilde\theta > v_1(x,a)} \\
&\le \sum_{(x,a)} \abs{\Phi_{(x,a),:} \widetilde\theta} \one{\mu_0(x,a)+\Phi_{(x,a),:} \widetilde\theta \le v_1(x,a)} + \norm{v_1}_1 \\ 
&\le \sum_{(x,a)} \norm{\Phi_{(x,a),:}} \norm{\widetilde\theta} \one{\mu_0(x,a)+\Phi_{(x,a),:} \widetilde\theta \le v_1(x,a)} + \norm{v_1}_1 \\ 
&\le \sum_{(x,a)} S C_1 q_1(x,a) \one{\mu_0(x,a)+\Phi_{(x,a),:} \widetilde\theta \le v_1(x,a)} + \norm{v_1}_1  \\
&\le S C_1 \epsilon_1 + \norm{v_1}_1 \; .
\end{align*}
\end{proof}

\begin{lemma}
\label{lem:sta_error}
Let $\delta_2\in (0,1)$ and $\epsilon_2\in (0,1)$. 
If we choose $k_2=\frac{4}{\epsilon_2}\left(d\log\frac{12}{\epsilon_2}+\log\frac{2}{\delta_2}\right)$, then with probability at least $1-\delta_2$, $\norm{(P-B)^\top\Phi \widetilde\theta}_1\leq S C_2 \epsilon_{2} + \norm{v_2}_1$. 
\end{lemma}
\begin{proof}
Applying theorem \ref{thm:vanroy}, we have that $q_2\left( \abs{(P-B)_{:,x}^\top \Phi \widetilde\theta} \le v_2(x) \right)\geq 1-\epsilon_{2}$. This yields
\beq
\label{eq:q2}
\sum_{x} q_2(x) \one{\abs{(P-B)_{:,x}^\top \Phi \widetilde\theta} \ge v_2(x)} \le \epsilon_{2} \; .
\eeq
Let $L' = \sum_x \abs{(P-B)_{:,x}^\top \Phi \widetilde\theta}$. 
Thus, with probability $1-\delta_2$,
\begin{align*}
L' &= \sum_x \abs{(P-B)_{:,x}^\top \Phi \widetilde\theta} \one{\abs{(P-B)_{:,x}^\top \Phi \widetilde\theta} > v_2(x)}\\ 
&\qquad+ \sum_x \abs{(P-B)_{:,x}^\top \Phi \widetilde\theta} \one{\abs{(P-B)_{:,x}^\top \Phi \widetilde\theta} \le v_2(x)} \\
&\le \sum_x \norm{(P-B)_{:,x}^\top \Phi} \norm{ \widetilde\theta} \one{\abs{(P-B)_{:,x}^\top \Phi \widetilde\theta} > v_2(x)} + \norm{v_2}_1  \\ 
&\le \sum_x S C_2 q_2(x) \one{\abs{(P-B)_{:,x}^\top \Phi \widetilde\theta} > v_2(x)} + \norm{v_2}_1  \\ 
&\le S C_2 \epsilon_{2} + \norm{v_2}_1 \,,
\end{align*}
where the last step follows from \eqref{eq:q2}. 

\end{proof}

We are ready to prove the main result of this section. Let $\widetilde \theta$ denote the solution of the sampled ALP, $\theta_{*}$ denote the solution of the full ALP \eqref{eq:slp2}, and $\mu_{\widetilde\theta}$ be the stationary distribution of the solution policy. Our goal is to compare $\ell^\top \mu_{\widetilde\theta}$ and $\ell^\top \mu_{\theta_*}$.
\begin{thm}
\label{thm:const-sampl}
Let $\epsilon\in(0,1)$ and $\delta\in(0,1)$. Let $\epsilon' = S C_1 \epsilon + \norm{v_1}_1$ and $\epsilon'' = S C_2 \epsilon + \norm{v_2}_1$. If we sample constraints with $k_1=\frac{4}{\epsilon}\left(d\log\frac{12}{\epsilon}+\log\frac{4}{\delta}\right)$ and $k_2=\frac{4}{\epsilon}\left(d\log\frac{12}{\epsilon}+\log\frac{4}{\delta}\right)$, then, with probability $1-\delta$,
\begin{align*}
\ell^\top \mu_{\widetilde\theta}  &\le \ell^\top \mu_{\theta_*} + \tau(\mu_{\widetilde\theta}) \log(1/\epsilon') (2\epsilon'+\epsilon'') + 3\epsilon' \\ 
&\quad+ \tau(\mu_{*}) \log(1/\norm{v_1}) (2\norm{v_1}+\norm{v_2}) + 3\norm{v_1}  \; .
\end{align*}
\end{thm}
\begin{proof}
Let $\delta_1=\delta_2=\delta/2$. By Lemmas~\ref{lem:sim_error} and \ref{lem:sta_error}, w.p. $1-\delta$, $\sum_{(x,a)} \abs{[\mu_0(x,a)+\Phi_{(x,a),:} \widetilde\theta]_{-}} \le \epsilon'$ and $\norm{(P-B)^\top(\mu_0+\Phi\widetilde\theta)}_1\leq\epsilon''$. Then by Lemma~\ref{lem:lem1}, 
\[
\abs{\ell^\top \mu_{\widetilde\theta} - \ell^\top (\mu_0+\Phi \widetilde\theta)} \le \tau(\mu_{\widetilde\theta}) \log(1/\epsilon') (2\epsilon'+\epsilon'') + 3\epsilon' \; .
\]
We also have that $\ell^\top (\mu_0+ \Phi \widetilde\theta) \le \ell^\top (\mu_0+\Phi \theta_{*})$. Thus,
\begin{align*}
\ell^\top \mu_{\widetilde\theta} &\le \ell^\top (\mu_0+\Phi\theta_{*}) + \tau(\mu_{\widetilde\theta}) \log(1/\epsilon') (2\epsilon'+\epsilon'') + 3\epsilon' \\
&\le \ell^\top \mu_{\theta_*} + \tau(\mu_{\widetilde\theta}) \log(1/\epsilon') (2\epsilon'+\epsilon'') + 3\epsilon' \\
&\quad+ \tau(\mu_{\theta_*}) \log(1/\norm{v_1}) (2\norm{v_1}+\norm{v_2}) + 3\norm{v_1} ,
\end{align*}
where the last step follows from Lemma~\ref{lem:lem1}.
\end{proof}
%Recall functions $V_1$ and $V_2$ as defined in Theorem~\ref{thm:main}. By definitions of $\theta_*$ and $v_1$ and $v_2$, it is easy to see that $V_1(\theta_*) \le \norm{v_1}_1$ and $V_2(\theta_*) \le \norm{v_2}_1$. On the other hand, the bound in Theorem~\ref{thm:main} scales with $1/\epsilon$, which is not the case in the above theorem. Thus, the two bounds are not directly comparable. 

\section{Experiments}

\begin{figure*}[t]\label{fig:4Dqueue}
\centering
\includegraphics[scale=1.5,clip=true]{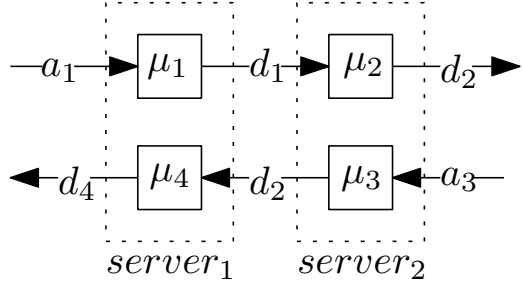}
\caption{The 4D queueing network. Customers arrive at queue $\mu_1$ or $\mu_3$ then are referred to queue $\mu_2$ or $\mu_4$, respectively. Server 1 can either process queue 1 or 4, and server 2 can only process queue 2 or 3.}
\end{figure*}

In this section, we apply both algorithms to the four-dimensional discrete-time queueing network illustrated in Figure \ref{fig:4Dqueue}. This network has a relatively long history; see, e.g. \cite{Rybko-Stolyar-1992} and more recently \cite{DeFarias-VanRoy-2003} (c.f. section 6.2). There are four queues, $\mu_1,\ldots,\mu_4$, each with state $0,\ldots,B$. Since the cardinality of the state space is $X=(1+B)^4$, even a modest $B$ results in huge state-spaces. For time $t$, let $X_t\in X$ be the state and $s_{i,t}\in\{0,1\}$, $i=1,2,3,3$ denote whether queue $i$ is being served. Server 1 only serves queue 1 or 4, server 2 only serves queue 2 or 3, and neither server can idle. Thus, $s_{1,t}+s_{4,t}=1$ and $s_{2,t}+s_{3,t}=1$. The dynamics are as follows. At each time $t$, the following random variables are sampled independently: $A_{1,t}\sim\text{Bernoulli}(a_1)$, $A_{3,t}\sim\text{Bernoulli}(a_3)$, and $D_{i,t}\sim\text{Bernoulli}(d_i*s_{i,t})$ for $i=1,2,3,4$. Using $e_1,\ldots,e_4$ to denote the standard basis vectors, the dynamics are:
\begin{align*}
X'_{t+1}=&X_t+A_{1,t}e_1+A_{3,t}e_3\\
&+D_{1,t}(e_2-e_1)-D_{2,t}e_2\\
&+D_{3,t}(e_4-e_3)-D_{4,t}e_4,
\end{align*}
and $X_{t+1}=\max(\mathbf{0},\min(\mathbf(B),X'_{t+1}))$ (i.e. all four states are thresholded from below by 0 and above by $B$). 
The loss function is the total queue size: $\ell(X_t)=||X_t||_1$. We compared our method against two common heuristics. In the first, denoted LONGER, each server operates on the queue that is longer with ties broken uniformly at random (e.g. if queue 1 and 4 had the same size, they are equally likely to be served). In the second, denoted LBFS (last buffer first served), the downstream queues always have priority (server 1 will serve queue 4 unless it has length 0, and server 2 will serve queue 2 unless it has length 0). These heuristics are common and have been used an benchmarks for queueing networks (e.g. \cite{DeFarias-VanRoy-2003}).

We used $a_1=a_3=.08$, $d_1=d_2=.12$, and $d_3=d_4=.28$, and buffer sizes $B_1=B_4=38$, $B_2=B_3=25$ as the parameters of the network.. The asymmetric size was chosen because server 1 is the bottleneck and tend to have has longer queues. The first two features are the stationary distributions corresponding to two heuristics.  We also included two types of non-stationary-distribution features. For every interval $(0,5],(6,10],\ldots,(45,50]$ and action $A$, we added a feature $\psi$ with $\phi(x,a)=1$ if $\ell(x,a)$ is in the interval and $a=A$. To define the second type, consider the three intervals $I_1=[0,10]$, $I_2=[11, 20]$, and $I_3=[21, 25]$. For every 4-tuple of intervals $(J_1,J_2,J_3,J_4)\in\{I_1,I_2,I_3\}^4$ and action $A$, we created a feature $\psi$ with $\psi(x,a)=1$ only if $x_i\in J_i$ and $a=A$. Every feature was normalized to sum to 1. In total, we had 372 features which is about a $10^{4}$ reduction in dimension from the original problem.

\subsection{Stochastic Gradient Descent}
\label{sec:stoch-grad-exps}

\begin{figure*}[t]
\centering
\label{fig:4Dplots}
\includegraphics[width=16cm]{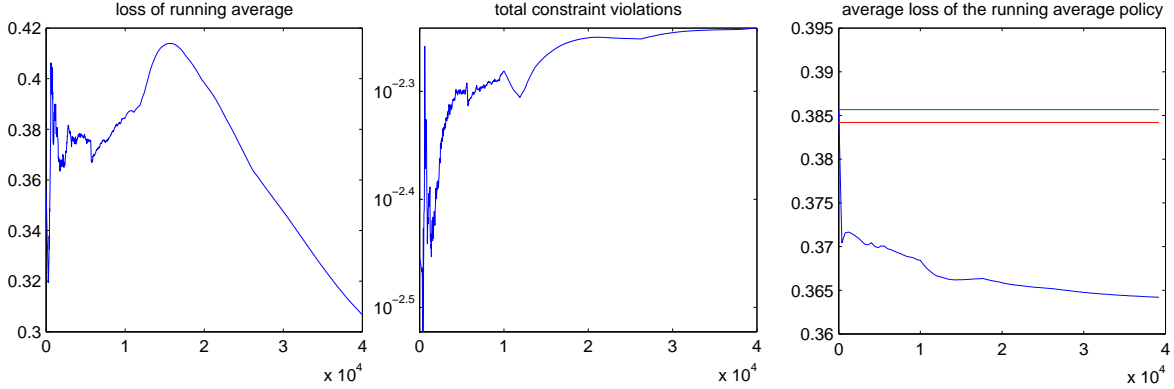}
\caption{ The left plot is of the linear objective of the running average, i.e. $\ell^\top \Phi\widehat\theta_t$. The center plot is the sum of the two constraint violations of $\widehat\theta_t$, and the right plot is $\ell^\top \tilde\mu_{\widehat\theta_t}$ (the average loss of the derived policy). The two horizontal lines correspond to the loss of the two heuristics, LONGER and LBFS.}
\end{figure*}

We ran our stochastic gradient descent algorithm with $I=1000$ sampled constraints and constraint gain $H=2$. Our learning rate began at $10^{-4}$ and halved every $2000$ iterations. The results of our algorithm are plotted in Figure \ref{fig:4Dplots}, where $\widehat\theta_t$ denotes the running average of $\theta_t$. The left plot is of the LP objective, $\ell^\top(\mu_0+\Phi\widehat\theta_t)$. The middle plot is of the sum of the constraint violations, $\norm{[\mu_0+\Phi\widehat\theta_t]_{-} }_1+ \norm{(P-B)^\top\Phi\widehat\theta_t }_1$. Thus, $c(\widehat\theta_t)$ is a scaled sum of the first two plots. Finally, the right plot is of the average losses, $\ell^\top\mu_{\widehat\theta_t}$ and the two horizontal lines correspond to the loss of the two heuristics, LONGER and LBFS. The right plot demonstrates that, as predicted by our theory, minimizing the surrogate loss $c(\theta)$ does lead to lower average losses. 

All previous algorithms (including \cite{DeFarias-VanRoy-2003}) work with value functions, while our algorithm works with stationary distributions. Due to this difference, we cannot use the same feature vectors to make a direct comparison. The solution that we find in this different approximating set is slightly worse than the solution of \citet{DeFarias-VanRoy-2003}.

\subsection{Constraint Sampling}
\label{sec:const-sampling-exps}

For the constraint sampling algorithm, we sampled the simplex constraints uniformly  with 10 different sample sizes: 508, 792, 1235, 1926, 3003, 4684, 7305, 11393, 17768, and 27712. Since $XA=4.1*10^6$, these sample sizes correspond to less that 1\%. The stationary constraints were sampled in the same proportion (i.e. $A$ times fewer samples). Let $a_1,\ldots,a_{AN}$ and $b_1,\ldots,b_N$ be the indices of the sampled simplex and stationary constraints, respectively. Explicitly, the sampled LP is:
\begin{align}
\label{eq:sampled_LP}
&\min_{\theta} (\Phi \theta)^\top  \ell\,, \\
\notag
&\mbox{s.t.}\quad (\Phi\theta)^\top  \mathbf 1 = 1,\, \Phi_{a_i,;} \theta \geq \mathbf 0,\, \forall i=1,\ldots,AN,\\
&\abs{\Phi\theta^\top (P - B)_{:,b_i}} \leq \mathbf \epsilon_s,\, \forall i=1,\ldots,N,\, \norm{\theta}_\infty\leq M
\end{align}
where $M$ and $\epsilon$ are necessary to ensure the LP always has a feasible  and bounded solution. This corresponds to setting $v_1=0$ and $v_2=\epsilon$. In particular, we used $M=3$ and $\epsilon=10^{-3}$. Using differenc values of $\epsilon$ did not have a large effect on the behavior of the algorithm.

For each sample size, we sample the constraints, solve the LP, then simulate the average loss of the policy. We repeated this procedure 35 times for each sample size and plotted the mean with error bars corresponding to the variance across each sample size in Figure \ref{fig:constraint_sampling}. Note the log scale on the x-axis. The best loss corresponds to 4684 sampled simplex constraints, or roughly 1\%, and is a marked improvement over the average loss found by the stochastic gradient descent method. However, changing the sample size by a factor of 4 in either direction is enough to obliterate this advantage. 
\begin{figure*}[h]
\centering
\label{fig:constraint_sampling}
\includegraphics[width=16cm]{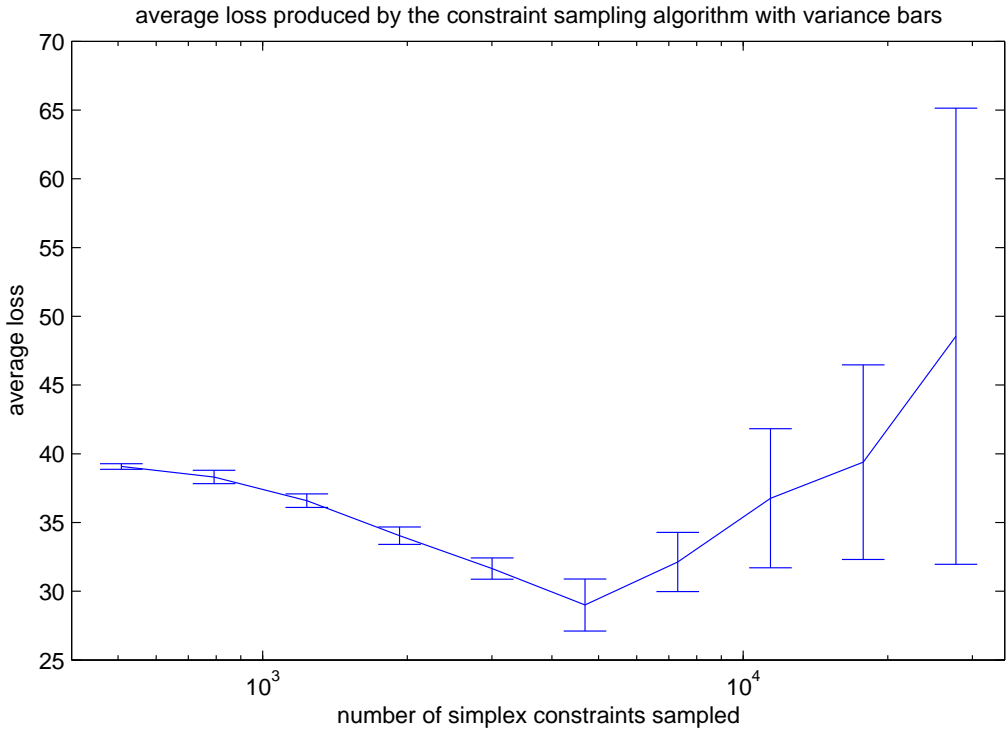}
\caption{Average loss with variance error bars of the constraint sampling algorithm run with a variety of sample sizes.}
\end{figure*}

First, we notice that the mean average loss is not monotonic. If we use too few constraints, then the sampled ALP does not reflect our original problem and we expect that the solution will make poor policies. On the other hand, if we sample too many constraints, then the LP is too restrictive and cannot adequately explore the feature space. To explain the increasing variance, recall that we have three families of constraints: the simplex constraints, the stationarity constraints, and the box constraints (i.e. $|\theta|_\infty\leq M$). Only the simplex and stationarity constraints are sampled. For the small sample sizes, the majority of the active constraints are the box constraint so $\tilde\theta$ (the minimizer of the LP) is not very sensitive to the random sample. However, as the sample size grows, so does the number of active simplex and stationarity constraints; hence, the random constraints affect  $\tilde\theta$ to a greater degree and the variance increases.

\section{Conclusions}
In this paper, we defined and solved the extended large-scale efficient ALP problem. We proved that, under certain assumptions about the dynamics, the stochastic subgradient method produces a policy with average loss competitive to all $\theta\in\Theta$, not just all $\theta$ producing a stationary distribution. We demonstrated this algorithm on the Rybko-Stoylar four-dimensional queueing network and recovered a policy better than two common heuristics and comparable to previous results on ALPs~\cite{DeFarias-VanRoy-2003}. %We also propose a constraint sampling method that has similar performance guarantees but under an additional assumption on the choice of features. 
A future direction is to find other interesting regularity conditions under which we can handle large-scale MDP problems. We also plan to conduct more experiments with challenging large-scale problems.

\section{Acknowledgements}

We gratefully acknowledge the support of the NSF through grant CCF-1115788 and of the ARC through an Australian Research Council Australian Laureate Fellowship (FL110100281).

\bibliography{all_bib}

\newpage
\onecolumn
\appendix

\end{document}